\documentclass[12pt]{amsart}

\usepackage{amsmath,amsthm,amssymb}

\usepackage[margin=1in]{geometry}

\newtheorem{thm}{Theorem}[section]
\newtheorem{cor}[thm]{Corollary}
\newtheorem{lem}[thm]{Lemma}
\newtheorem{ex}[thm]{Example}

\theoremstyle{definition}
\newtheorem{defi}[thm]{Definition}

\newtheorem*{notation}{Notation}

\usepackage{pgf,tikz}
\usetikzlibrary{arrows,shapes}
\usetikzlibrary{decorations.pathreplacing}
\usepackage{tkz-berge}
\usepackage{ytableau}
\usepackage{multicol}
\usepackage{hyperref}
\usepackage{todonotes}
\usepackage{mathtools}

\newcommand{\ov}{\overline}
\newcommand{\und}{\underline}

\newcommand{\M}{\mathcal{M}}
\newcommand{\R}{\mathbb{R}}
\newcommand{\N}{\mathbb{N}}
\newcommand{\cL}{\mathcal{L}}
\newcommand{\cC}{\mathcal{C}}
\newcommand{\cO}{\mathcal{O}}
\newcommand{\cG}{\mathcal{G}}
\newcommand{\cB}{\mathcal{B}}

\newcommand\setof[2]{\left\{#1\;:\;#2\right\}}
\DeclarePairedDelimiter\set{\{}{\}}
\DeclarePairedDelimiter\abs{\lvert}{\rvert}

\DeclareMathOperator{\im}{im} 
\DeclareMathOperator{\ind}{ind}

\title{Graphs with the Fewest Matchings}
\author{L. Keough}
\author{A.J. Radcliffe}

\begin{document}

\begin{abstract}
In recent years there has been increased interest in extremal problems for ``counting'' parameters of graphs.  For example, the Kahn-Zhao theorem gives an upper bound on the number of independent sets in a $d$-regular graph.  In the same spirit, the Upper Matching Conjecture claims an upper bound on the number of $k$-matchings in a $d$-regular graph. Here we consider both matchings and matchings of fixed sizes in graphs with a given number vertices and edges. We prove that the graph with the fewest matchings is either the lex or the colex graph.  Similarly, for fixed $k$, the graph with the fewest $k$-matchings is either the lex or the colex graph. To prove these results we first prove that the lex bipartite graph has the fewest matchings of all sizes among bipartite graphs with fixed part sizes and a given number of edges.
\end{abstract}

\maketitle

\section{Introduction and statement of results}

In recent years there has been increased interest in extremal problems for ``counting'' parameters of graphs. A classic example is the Kahn-Zhao theorem, proved initially by Kahn \cite{Kahn} in the bipartite case, and then extended to the general case by Zhao \cite{Zhao}.

\begin{thm}[Kahn-Zhao]\label{thm:KahnZhao}
    If $G$ is a $d$-regular graph then $\ind(G)$, the number of independent sets in $G$, satisfies
    \[
        \ind(G) \le \left(2^{d+1}-1\right)^{\frac{n}{2d}} = (\ind(K_{d,d}))^{\frac{n}{2d}}.
    \]
\end{thm}

Such problems have been studied (extensively) for the number of perfect matchings in a graph: see for instance \cite{Alon, KahnCuckler,EntropyKahnLovasz, Gross} and others. Usually these results concern regular graphs or graphs with a given degree sequence. In another vein there has been work on determining which graph with given numbers of vertices and edges maximizes or minimizes a given counting parameter. (We write $\cG_{n,e}$ for the class of graphs having $n$ vertices and $e$ edges.) For instance Cutler and Radcliffe \cite{WidomRowlinson,TheFox} 
determined which graphs have the largest number of homomorphisms in various fixed images graphs $H$. They also showed that the Kruskal-Katona Theorem implies that the lex graph has the greatest number of independent sets among graphs in $\cG_{n,e}$.  A classic paper of Ahlswede and Katona \cite{Ahlswede} determines the minimum number of pairs in non-incident edges a graph in $\cG_{n,e}$ can have. (Which is of course the same as maximizing the number of pairs of incident edges.)

In this paper we solve the problem of determining which graph with
$n$ vertices and $e$ edges has the fewest matchings. (A
\emph{matching} in a graph is simply a set of pairwise non-incident
edges.) We denote the set of matchings in $G$ by $\M(G)$ and the set
of $k$-matchings in $G$ by $\M_k(G)$.  Also we write $m(G)=|\M(G)|$ and
$m_k(G)= |\M_k(G)|$. In general our notation is standard and
follows that of Bollob\'as \cite{MGT}. It turns out that, following
the general approach of Ahlswede and Katona, we need to consider the
class $\cB_{\ell,r,e}$ of all bipartite graphs with $\ell$ vertices
in the left part, $r$ vertices on the right, and having $e$ edges.
We determine the graph in $\cB_{\ell,r,e}$ having the fewest
matchings. Our techniques also allow us to determine the graphs
minimizing the number of matchings of size $k$, for all values of
$k$, in $\cB_{\ell,r,e}$ and in $\cG_{n,e}$. In order to state our
results we need to describe three collections of graphs: lex graphs,
colex graphs, and lex bipartite graphs.

\begin{defi}  The \emph{lexicographic order}, $<_L$, on finite subsets of $\N$ is defined by $A<_L B$ if  $\min(A\Delta B)\in A$. The \emph{colexigraphic order}, $<_C$, is defined by $A<_C B$ if $\max(A\Delta B) \in B$.
\end{defi}

Restricting these orderings to $2$-subsets of $[n]$ results in the lex and colex orderings on $E(K_n)$.    The first few edges in the lex ordering on $E(K_n)$ are
$$\set{1,2},\set{1,3},\dots,\set{1,n},\set{2,3},\set{2,4},\dots,\set{2,n},\set{3,4},\dots$$
and the first few edges in the colex ordering on $E(K_n)$ are
$$\set{1,2},\set{1,3},\set{2,3},\set{1,4},\set{2,4},\set{3,4},\set{1,5},\set{2,5},\dots$$

\begin{defi}
The \emph{lex graph} $\cL(n,e)$ is the graph with vertex set $[n]$ and edge set consisting of the first $e$ edges in the lex order on $E(K_n)$.  Similarly,  the \emph{colex graph} $\cC(n,e)$ is the graph with vertex set $[n]$ and edge set consisting of the first $e$ edges in the colex order on $E(K_n)$.
\end{defi}

\begin{ex}  The graph below is $\cL(7,8)$.
\begin{center}
\begin{tikzpicture}[scale=.7]
    \coordinate (0) at (0,0) ;
    \coordinate (1) at (1,0) ;
    \coordinate (2) at (2,0) ;
    \coordinate (3) at (3,0) ;
    \coordinate (4) at (4,0) ;
    \coordinate (5) at (5,0) ;
    \coordinate (6) at (6,0) ;

    \fill (0) circle (3pt) ;
    \fill (1) circle (3pt) ;
    \fill (2) circle (3pt) ;
    \fill (3) circle (3pt) ;
    \fill (4) circle (3pt) ;
    \fill (5) circle (3pt) ;
    \fill (6) circle (3pt) ;

   \tikzstyle{EdgeStyle}=[bend left]

   \Edge(0)(1)

   \Edge(0)(2)

   \Edge(0)(3)

   \Edge(0)(4)
   \Edge(0)(5)
   \Edge(0)(6)

   \Edge(1)(2)

   \Edge(1)(3)
\end{tikzpicture}
\end{center}
\end{ex}

\begin{ex}The graph below is $\cC(7,8)$.

\begin{center}
\begin{tikzpicture}[scale=.7]
    \coordinate (0) at (0,0) ;
    \coordinate (1) at (1,0) ;
    \coordinate (2) at (2,0) ;
    \coordinate (3) at (3,0) ;
    \coordinate (4) at (4,0) ;
    \coordinate (5) at (5,0) ;
    \coordinate (6) at (6,0) ;

    \fill (0) circle (3pt) ;
    \fill (1) circle (3pt) ;
    \fill (2) circle (3pt) ;
    \fill (3) circle (3pt) ;
    \fill (4) circle (3pt) ;
    \fill (5) circle (3pt) ;
    \fill (6) circle (3pt) ;

   \tikzstyle{EdgeStyle}=[bend left]

   \Edge(0)(1)

   \Edge(0)(2)

   \Edge(1)(2)

   \Edge(0)(3)
   \Edge(1)(3)
   \Edge(2)(3)

   \Edge(0)(4)

   \Edge(1)(4)
\end{tikzpicture}
\end{center}
\end{ex}

Our main result is that $m(G)$ and $m_k(G)$ are each minimized by either the lex graph or the colex graph.

\begin{thm}\label{general}
For all graphs $G$ with $n$ vertices and $e$ edges, and for all $k$,
$$m(G) \geq \min\{ m(\cL(n,e)), m(\cC(n,e)\}$$
and
$$m_k(G) \geq \min\{ m_k(\cL(n,e)), m_k(\cC(n,e)\}.$$
\end{thm}

\begin{defi}
Suppose $n = \ell+ r$ with $\ell\leq r$ and $e\leq \ell r$.  Write $e = qr+c$ where $0\leq c < r$.  The \emph{lex bipartite graph} with $e$ edges and partite sets $L$ and $R$ of size $\ell$ and $r$ respectively is the bipartite graph in which $q$ vertices in $L$ have degree $r$ and one vertex in $L$ has degree $c$.  We will denote this graph by $L_{\ell,r}(e)$. Note that the lex bipartite graph contains the first $e$ edges of $E(L,R)$ in lex order.
\end{defi}

A core part of the proof of Theorem~\ref{general} is to establish that, for all $k\geq 0$, $L_{\ell,r}(e)$ minimizes $m_k(G)$ in the class $\cB_{\ell,r,e}$.  It is key that there is a unique minimizer in the bipartite case.

\begin{thm}\label{bipartite}
Suppose $1\leq k\le \ell \leq r$ and $B\in \cB_{\ell,r,e}$.  Then
$$m(B)\geq m(L_{\ell,r}(e))$$
and
$$m_k(B)\geq m_k(L_{\ell,r}(e)).$$
\end{thm}

To prove Theorem~\ref{general}, we will first show that the graph attaining the minimum number of matchings is threshold.  In Section~\ref{Threshold} we discuss the results we need concerning threshold graphs. In Section 3 we prove Theorem~\ref{bipartite}, i.e., that there is a bipartite graph that simultaneously minimizes the number of matchings of each size. Finally, we use the bipartite case to show that the lex or colex graph minimizes $m(G)$ and $m_k(G)$ in the family $\cG_{n,e}$ in Section 4.

\section{Threshold Graphs}\label{Threshold}

Threshold graphs appear as an answer to many extremal questions, especially those in which it is advantageous to have all the edges ``bunched together''.
For instance in the class $\cG_{n,e}$ threshold graphs maximize the number of independent sets and minimize homomorphisms into the Widom-Rowlinson graph \cite{WidomRowlinson}.

There are many equivalent definitions of threshold graphs. The one that gives them their name is as follows.
\begin{defi}
A simple graph $G$ is a \emph{threshold graph} if there exists a function $w: V(G) \to \R$ and a threshold $t\in \R$ such that $xy\in E(G)$ if and only if $w(x)+w(y)\geq t$.
\end{defi}

The following lemma describes an alternative characterization that we will use in the proof of Theorem~\ref{general}.
\begin{lem}[\cite{OrangeBook}]
A graph $G$ is a threshold graph if and only if $V(G)$ can be partitioned into a clique and an independent set and moreover there exists a labeling $i_1,\dots, i_k$ of the vertices in the independent set such that $N(i_1) \supseteq N(i_2) \supseteq \dots \supseteq N(i_k)$. $\hfill$ $\blacksquare$
\end{lem}

Note that it is immediate in this definition to see that the lex and colex graphs are threshold. In the colex case only (at most) one vertex in the independent set has any neighbors at all. In the lex graph case the clique consists (with one possible exception) of dominant vertices, so all the vertices in the independent set are joined either to all the vertices in the clique or all but one (and the one missing vertex is the same in each case).
The following equivalent characterization of threshold graphs will help us define a way to move a graph towards being threshold.
\begin{lem}[\cite{OrangeBook}]\label{thresholdnbhd}
A graph $G$ is threshold if and only if for all $x,y\in V(G)$ we have $N[x] \subseteq N[y]$ or $N[y]\subseteq N[x]$. $\hfill$ $\blacksquare$
\end{lem}

We now define a compression move that makes a graph ``more
threshold".  We will use this move later to show that we can find a
graph that minimizes $m_k(G)$ in $\cG_{n,e}$ that is threshold.

\begin{defi}
Let $G$ be a graph and $x$ and $y$ two vertices in $G$.  Define
\[
N_G(x,\ov{y}) =\{v\in V(G)\setminus \{x,y\}: v\sim x,v\not\sim y\}.
\]
 Let $G_{x\to y}$ be the graph formed by deleting all edges between $x$ and $N_G(x,\ov{y})$ and adding all edges from $y$ to $N_G(x,\ov{y})$.  This is called the \emph{compression of $G$ from $x$ to $y$}. It is clear that $G_{x\to y}$ has the same number of edges as $G$.
\end{defi}

We will begin our proof of Theorem~\ref{general} by repeatedly compressing a graph that minimizes the number of matchings. The following lemma will allow us to be sure that we are making progress, and not compressing round and round in a circle. The variance of the degree sequence, or (essentially equivalently) the quantity
\[
    d_2(G) = \sum_{v\in V(G)} d(v)^2
\]
strictly increases whenever we do a non-trivial compression.

\begin{lem}[\cite{WidomRowlinson}]\label{lem:allcompressed}
Given $x,y\in V(G)$, $e(G_{x\to y}) = e(G)$ and $d_2(G_{x\to y})\geq d_2(G)$.  If $N[x]\not\subseteq N[y]$ and $N[y] \not\subseteq N[x]$ then $d_2(G_{x\to y}) > d_2(G)$. $\hfill$ $\blacksquare$
\end{lem}

\begin{cor}\label{cor:allcompressed2}
 Suppose that $\cG$ is a family of graphs  on a fixed vertex set $V$ such that for any $G'\in\cG$ and $x,y\in V$ we also have $G'_{x\to y}\in \cG$.  In addition suppose that $G$ satisfies
\[
    d_2(G) = \max\setof{d_2(G')}{G'\in \cG}.
\]
Then $G$ is threshold.
\end{cor}
\begin{proof}
 Suppose $x,y\in V$ and $N_G[x]$ and $N_G[y]$ are incomparable.  By hypothesis, $G_{x\to y}\in \cG$ and by Lemma~\ref{lem:allcompressed} we know $d_2(G_{x\to y}) > d_2(G)$.  This contradicts the assumption that $G$ attains the maximum value of $d_2$ in $\cG$.  Thus, $N_G[x] \subseteq N_G[y]$ or $N_G[y]\subseteq N_G[x]$ and so $G$ is threshold by Lemma~\ref{thresholdnbhd}.
 \end{proof}

 We will use the closely related topic of threshold bipartite graphs to prove Theorem~\ref{bipartite}.
 In \cite{OrangeBook} threshold bipartite graphs are defined with a similar vertex weighting.
 \begin{defi}\label{thresholdbipartite}  A graph $G=(V,E)$ is said to be \emph{threshold bipartite} if there exists a threshold $t$ and a function $w: V(G) \to \R$ such that  $\abs{w(v)}<t$ for all $v\in V$ and distinct vertices $u$ and $v$ are adjacent if and only if $\abs{w(u)-w(v)}\geq t$.
\end{defi}

Threshold bipartite graphs are called difference graphs in \cite{OrangeBook} and chain graphs in \cite{Yannakakis}.  As with threshold graphs, there are many equivalent definition of threshold bipartite graphs.
The following lemma describes the definition that will be most useful to us.

\begin{lem}[\cite{Hammer}]\label{bipartiteinclusion}
    A graph is threshold bipartite if and only if $G$ is bipartite and the neighborhoods of vertices in one of the partite sets can be linearly ordered by inclusion. $\hfill$ $\blacksquare$
\end{lem}

Note that a threshold graph can be obtained from a threshold bipartite graph by adding all possible edges in one of the partite sets (on either side).

\begin{lem}\label{bipartited2}
Let $G$ be a bipartite graph with bipartition $(X,Y)$.  Given $u,v\in X$, $e(G_{u\to v}) = e(G)$ and $d_2(G_{u\to v})\geq d_2(G)$.  If $N(u)\not\subseteq N(v)$ and $N(u) \not\subseteq N(v)$ then $d_2(G_{u\to v}) > d_2(G)$.
\end{lem}
\begin{proof}
Same calculations as in the proof of Lemma~\ref{lem:allcompressed}.
\end{proof}

 \begin{cor}\label{bipartitecompress}
 Suppose that $\cG$ is a family of bipartite graphs on a fixed vertex set $V$ with a fixed bipartition $(X,Y)$ such that for any $G'\in\cG$ and $u,v\in X$ we also have $G'_{u\to v}\in \cG$.  In addition suppose that $G$ satisfies
 $$d_2(G) = \max\{d_2(G'): G'\in \cG\}.$$
 Then $G$ is bipartite threshold.
 \end{cor}
 \begin{proof}
 Suppose that $u,v\in X$ such that $N_G(u) \not\subseteq N_G(v)$ and $N_G(v) \not\subseteq N_G(u)$.  Then $G_{u\to v} \in \cG$ by assumption and by Lemma~\ref{bipartited2} $d_2(G_{u\to v}) > d_2(G)$, a contradiction.  Thus, $N_G(u) \subseteq N_G(v)$ or $N_G(v) \subseteq N_G(u)$ and so by Lemma~\ref{bipartiteinclusion} the graph $G$ is threshold bipartite.
 \end{proof}

\subsection{Partitions}

We can relate threshold bipartite graphs to partitions and matchings in threshold bipartite graphs to rook placements in the Young diagram of the partition.  We will make use of this connection heavily in the proof of Theorem~\ref{bipartite}.

Given $(\ell,r,e)$, threshold bipartite graphs $G$ with vertex classes of size $\ell$ and $r$ and $\abs{E(G)} = e$ are in bijective correspondence with partitions of the integer $e$ with at most $\ell$ parts each of size at most $r$. From a partition $\lambda$ we can construct the associated bipartite graph $G_{\lambda}$ by letting $E(G_{\lambda}) = \{x_iy_j: j\leq \lambda_i\}$.  Given a threshold bipartite graph we get a partition of $\abs{E(G)}$ by letting the degree of each vertex on the left be the size of a part. We will use this correspondence to represent threshold bipartite graphs as Young diagrams.

\begin{defi}
Let $B$ be a subset of $[\ell]\times [r]$. If $(i,j)\in B$ we call $(i,j)$ a \emph{box} of $B$.  We call $B$ a \emph{Young diagram} if for all $(i,j)\in B$ with $i>1$ we have $(i-1,j)\in B$ and for all $(i,j)\in B$ with $j>1$ we have $(i,j-1)\in B$. We call $[\ell]
\times [r]$ the \emph{frame} of the Young diagram and we'll say that
the Young diagram has dimensions $\ell\times r$. A \emph{matching} in a
Young diagram $B$ is a subset $M$ of $B$ such that for all
$(a_1,b_1),(a_2,b_2)\in M$ we have $a_1\neq a_2$ and $b_1\neq b_2$.
\end{defi}

Equivalently, a matching is a placement of non-attacking rooks on $B$.
(A placement of non-attacking rooks is a placement of rooks such that no two rooks are in the same row or column.)  The total number of ways to place non-attacking rooks is called the rook number.  There is extensive literature on rook numbers, for example see \cite{Rooknumber}.  We will use the language of rook placements in some of the proofs.

Matchings in a Young diagram $B$ correspond to a matchings in the bipartite graph associated with $B$ by equating the box $(i,j)\in B$ with the edge $x_iy_j$ in the associated bipartite graph.

\begin{notation}
Let the set of $k$-matchings in a Young diagram $B$ be denoted $\M_k(B)$ and let $\M(B) = \bigcup_{k\geq 0} \M_k(B)$.
Also define $m_k(B)= \abs{\M_k(B)}$ and $m(B) = \abs{\M(B)}$.
\end{notation}

\begin{ex}  Figure \ref{YD} is an example of a Young diagram in a $4\times 6$ frame and Figure 2 is a matching of size 3 in that Young diagram.  In all our diagrams we label rows and columns using matrix numbering.

\begin{figure}[!h]
\begin{multicols}{2}
\begin{center}

\begin{tikzpicture}[scale=.5]
   \draw (0,2) grid (6,3);
   \draw (0,1) grid (5,2);
   \draw (0,0) grid (3,1);
   \draw (0,-1) grid (2,0);
   \draw[color=gray] (0,-1) rectangle (6,3);
   \draw (0.5,3.5) node {$1$};
   \draw (1.5,3.5) node {$2$};
   \draw (2.5,3.5) node {$3$};
   \draw (3.5,3.5) node {$4$};
   \draw (4.5,3.5) node {$5$};
   \draw (5.5,3.5) node {$6$};
   \draw (-0.5,2.5) node {$1$};
   \draw (-0.5,1.5) node {$2$};
   \draw (-0.5,0.5) node {$3$};
   \draw (-0.5,-0.5) node {$4$};
\end{tikzpicture}
\caption{Young diagram of (6,5,3,2).}\label{YD}

\end{center}

\begin{center}
\begin{tikzpicture}[scale=.5]
   \draw (0,2) grid (6,3);
   \draw (0,1) grid (5,2);
   \draw (0,0) grid (3,1);
   \draw (0,-1) grid (2,0);
   \fill [color=black] (.5,.5) circle (5pt);
   \fill [color=black] (2.5,2.5) circle (5pt);
   \fill [color=black] (4.5,1.5) circle (5pt);
   \draw[color=gray] (0,-1) rectangle (6,3);
    \draw (0.5,3.5) node {$1$};
   \draw (1.5,3.5) node {$2$};
   \draw (2.5,3.5) node {$3$};
   \draw (3.5,3.5) node {$4$};
   \draw (4.5,3.5) node {$5$};
   \draw (5.5,3.5) node {$6$};
   \draw (-0.5,2.5) node {$1$};
   \draw (-0.5,1.5) node {$2$};
   \draw (-0.5,0.5) node {$3$};
   \draw (-0.5,-0.5) node {$4$};
\end{tikzpicture}
\caption{The matching $\{(1,3),(2,5),(3,1)\}$.}
\end{center}

\end{multicols}
\end{figure}
\end{ex}


\section{Bipartite Minimizer}
In this section we prove Theorem~\ref{bipartite}.  We are motivated to look at the bipartite case by the following lemma.

\begin{lem}\label{minimize}  Given a threshold graph $G$ with vertex set $V$, let $V = K \cup (V\setminus K)$ be a partition of the vertex set such that $G[K]$ forms a clique and $G[V\setminus K]$ forms an independent set.  Suppose $\abs{K} = s$.   Let $B$ be the bipartite graph with partite sets $K$ and $V\setminus K$ and edge set $E(K, V\setminus K)$.  Then
$$m(G) = \sum_{k\geq 0} m_k(B) \cdot m(K_{s-k}).$$
\end{lem}
\begin{proof}
Fixing a matching $M$ in $B$ of size $k$, there are exactly $m(K_{s-k})$ matchings in $G$ that contain $M$.  So there are $m_k(B) \cdot m(K_{s-k})$ matchings in $G$ that contain a $k$-matching in $B$.  Summing over all possible sizes of matchings in $B$ we count every matching in $G$ exactly once.
\end{proof}

Theorem~\ref{bipartite} states that for a given $k$ there is a graph
that simultaneously minimizes every $m_k(G)$ in $\cB_{\ell,r,e}$. We
will first show that there is a threshold bipartite graph that
minimizes $m_k(G)$ in $\cB_{\ell,r,e}$.  We will then use moves on
the associated Young diagrams to show that the lex bipartite graph
$L_{\ell,r}(e)$ minimizes $m_k(G)$ for all $k$.

\begin{lem}\label{differencegraph}
Suppose $G$ is a bipartite graph with bipartition $(X,Y)$.  Let $u,v \in X$. Then  for every $k\geq 0$ the graph $H:= G_{u \to v}$ has at most as many $k$-matchings as $G$.
\end{lem}
\begin{proof}
We will construct an injection from $\M(H)\setminus \M(G)$ to $\M(G)\setminus \M(H)$ that preserves size.  It then follows that $m_k(H)\leq m_k(G)$ for all $k$.  We define a replacement function $r: E(H) \to E(G)$ by
$$r(e) = \begin{cases}
uy &\text{ if } e=vy \text{ for } y \in N_G(u,\ov{v})\\
vz &\text{ if } e=uz\\
e  &\text{otherwise}
\end{cases}.$$

Given $e$ in the edge set of $H$, we claim that $r(e)$ is an edge in $G$. If $y\in N_G(u,\ov{v})$, then $uy\in E(G)$.  Also, if $uz\in H$ then $z\in N_G(u)\cap N_G(v)$ and so $vz\in E(G)$. Finally, if $e\neq vy$ for $y\in N_G(u\ov{v})$ then $e\in E(G)\cap E(H)$.

Now define $\phi: \M(H)\setminus\M(G) \to \M(G)\setminus\M(H)$ by
$$\phi(M) = \{r(e): e\in M\}$$

Given $M\in \M(H)\setminus\M(G)$ note that $\phi(M) \subseteq E(G)$ since $r(e)\in E(G)$ for all $e\in E(H)$.
We claim that in fact $\phi(M)\in \M(G)\setminus\M(H)$.
For the first case, suppose that $u\in r(e) \cap r(f)$.
Note that if $u\in r(e)$ then $v\in e$ since edges in $H$ containing $u$ are replaced by edges in $G$ containing $v$ instead.
So if $r(e) \cap r(f) = u$ then $e\cap f = v$, a contradiction since $e,f\in M$, a matching.

Now suppose that $e, f\in M$ and $r(e) \cap r(f) = v$. There are two
possible ways for $v$ to be in $r(e)$.  The first is for $u$ to be
in $e$ and the second is for $e=va$ for some $a\notin
N_G(u,\ov{v})$. However, since $M\in \M(H)\setminus\M(G)$ it must be
the case that $vy\in M$ for some $y \in N_G(u,\ov{v})$. Since $M$ is
a matching, $vy$ is the only edge incident to $v$. Therefore, $e$
and $f$ must both contain $u$, a contradiction since $M$ is a
matching.

Finally suppose that $z\in r(e) \cap r(f)$ for some $z\neq u,v$.
Then $z\in e\cap f$, a contradiction.
Thus each vertex has at most one incident edge and $\phi(M)$ is a matching.
Note that $\phi(M)\notin \M(G)$ since $uy\in \phi(M)$ for some $y\in N_G(u,\ov{v})$ and no such edge is in $E(G)$.  So $\phi(M) \in \M(H)\setminus \M(G)$.\\

To finish the proof of the lemma we need only show that $\phi$ is an
injection.   We'll show that $\phi$ has a left inverse defined
similarly to $\phi$.  Consider $r': E(G) \to E(H)$ defined by
$$r'(e)= \begin{cases}
vy &\text{ if } e = uy \text{ for any } y \in N_G(u,\ov{v})\\
uz &\text{ if } e = vz\\
e &\text{otherwise}
\end{cases}.
$$
Define $\phi': \M(G)\setminus \M(H) \to \M(H)\setminus \M(G)$ by $\phi'(M) = \{r'(e): e\in M\}$.  It is straightforward to check that $\phi'(\phi(M))= M$.  Thus $\phi$ has a left inverse and so $\phi$ is injective.
\end{proof}

By Corollary~\ref{bipartitecompress} and Lemma~\ref{differencegraph}, a bipartite graph minimizing the total number of matchings can be found among the threshold bipartite graphs.

\begin{defi}
Say $P$ is an \emph{out-corner} of a Young diagram $B$ if $P \in B$
and there is no box in the diagram to its right or beneath it. Say
$Q$ is an \emph{in-corner} if a box can be added there to create an
out-corner.
\end{defi}

\begin{figure}[h!]\label{corners}
\begin{center}
\begin{tikzpicture}[scale = .65]
   \draw (0,2) grid (6,3);
   \draw (0,1) grid (5,2);
   \draw (0,0) grid (3,1);
   \draw (0,-1) grid (2,0);
   \draw (0,3) rectangle (6,-2);
   \draw (5.5,2.5) node {O};
   \draw (4.5,1.5) node{O};
   \draw (2.5,.5) node{O};
   \draw (1.5,-.5) node{O};
   \draw (5.5, 1.5) node {I};
   \draw (3.5,.5) node {I};
   \draw (2.5,-.5) node {I};
   \draw (.5,-1.5) node {I};
\end{tikzpicture}

\caption{In-corners and out-corners of (6,5,3,2) in a $6\times 5$ frame}
\end{center}
\end{figure}
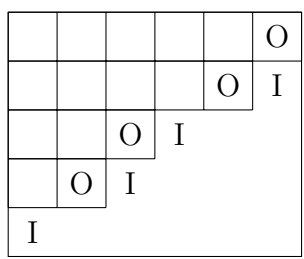

\begin{ex}  In Figure \ref{corners} the in-corners are labeled with $I$ and the out-corners are labeled with $O$.  Notice that the dimensions of the frame matter; for example, in a $4\times 6$ frame $(5,1)$ would not be an in-corner.
\end{ex}

We now define a move that removes a box that is an out-corner and puts a box at an in-corner that is ``further out".  Let $s(P)$ be the sum of the coordinates of $P$.

\begin{lem}\label{oneblock}
Let $B$ be a Young diagram in an $m\times n$ frame.  Suppose $P$ is
an out-corner of $B$ and $P'$ is an in-corner of $B$.  If $s(P)<
s(P')$ then $B':=B+P'-P$ is a Young diagram in an $m\times n$ frame
and has at most as many $k$-matchings as $B$ for all $k\geq 0$.
Moreover, $m(B')<m(B)$.
\end{lem}
\begin{proof}
Suppose $P=(i,j)$ and $P'=(i',j')$. The hypothesis states that $i+j<i'+j'$.  Define $B^+ = B+P' = B'+P$.  Note that $B^+$ is a Young diagram since $P'$ is an in-corner.  Showing that $B'$ has fewer matchings than $B$ is equivalent to showing that there are fewer matchings in $B^+$ that contain $P'$ and not $P$ than matchings in $B^+$ that contain $P$ and not $P'$. To do this we will define an injection from the collection of $k$-matchings of $B^+$ that contain $P'$ to the collection of $k$-matchings of $B^+$ that contain $P$.

Define $S:=\setof{(a,j)}{i'<a<i}$ and $T:=\setof{(i',b)}{ j<b<j'}$.  The injection will be defined in two parts: firstly on those matchings of $B$ that don't intersect $S$ and secondly on those that do.

\begin{center}
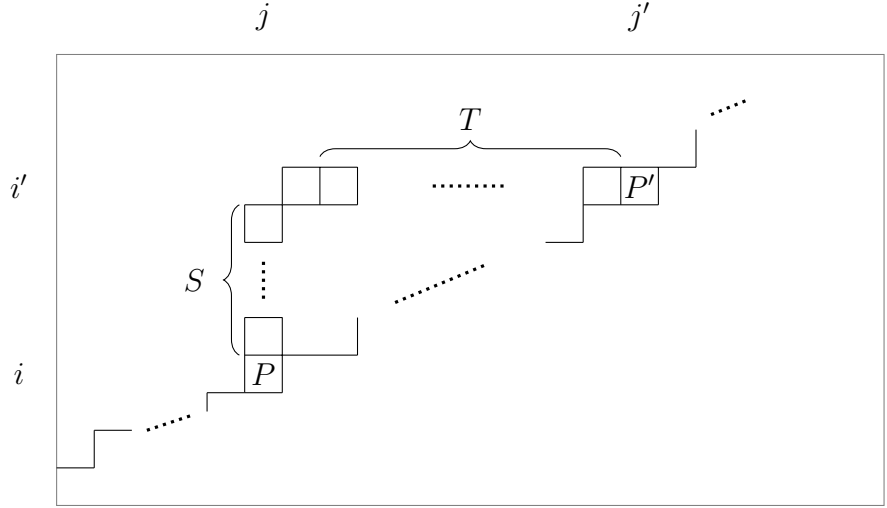
\begin{figure}[!h]
\begin{tikzpicture}
\draw[color=gray] (-1,0) rectangle (10,6);

\draw (-1.5, 1.75) node {$i$};
\draw (1.75 , 6.5) node {$j$};
\draw (-1.5, 4.25) node {$i'$};
\draw (6.75, 6.5) node {$j'$};

\draw[step=.5] (2,4) grid (3,4.5);
\draw (2,4) -- (3,4);
\draw (2,4) -- (2,4.5);
\draw[dotted,very thick] (4,4.25) -- (5,4.25);
\draw[step=.5] (6,4) grid (7,4.5);
\draw (6,4) -- (6,4.5);
\draw (6,4) -- (7,4);
\draw [decorate,decoration={brace,amplitude=6pt},xshift=0pt,yshift=4pt] (2.5,4.5) -- (6.5,4.5) node [black,midway,yshift=.5cm]{$T$};
\draw (6.75,4.25) node {$P'$};

\draw[step=.5] (1.5,3.5) grid (2,4);
\draw (1.5, 3.5) -- (1.5,4);
\draw (1.5,3.5) -- (2,3.5);
\draw[dotted,very thick] (1.75,2.75) -- (1.75,3.25);
\draw[step=.5] (1.5,1.5) grid (2,2.5);
\draw (1.5,1.5) -- (1.5,2.5);
\draw (1.5,1.5) -- (2,1.5);
\draw [decorate,decoration={brace,amplitude=6pt},xshift=-2pt,yshift=0pt] (1.5,2) -- (1.5,4) node [black,midway,xshift=-0.6cm]{$S$};
\draw (1.75,1.75) node {$P$};

\draw (-1,.5) -- (-.5,.5);
\draw (-.5,.5) -- (-.5,1);
\draw (-.5,1) -- (0,1);
\draw[dotted,very thick] (0.2,1) -- (.8,1.2);
\draw (1,1.5) -- (1.5,1.5);
\draw (1,1.25) -- (1,1.5);

\draw (2,2) -- (3,2);
\draw (3,2) -- (3,2.5);
\draw[dotted, very thick] (3.5,2.7) -- (4.7,3.2);
\draw (6,3.5) -- (6,4);
\draw (5.5,3.5) -- (6,3.5);

\draw (7,4.5) -- (7.5,4.5);
\draw (7.5,4.5) -- (7.5,5);
\draw[dotted, very thick] (7.7, 5.2) -- (8.2,5.4);

\end{tikzpicture}
\caption{The Young diagram $B^+$.}\label{boardB+}
\end{figure}
\end{center}

In Figure \ref{boardB+} we have drawn $P'$ up and to the right of $P$.    It is also possible that $P'$ is down and to the left of $P$.  In this case the same proof will work by using the transpose of $B^+$.

Let
 $$A_{\ov{S}}: = \{M\in M(B^+): P'\in M, P \notin M, S\cap M=\emptyset\},$$
 $$A_{\ov{T}}: = \{M\in M(B^+): P\in M, P'\notin M T\cap M=\emptyset\},$$
 $$A_{S}: = \{M\in M(B^+): P'\in M, P\notin M S\cap M\neq\emptyset\},$$
 $$A_{T}: = \{M\in M(B^+): P\in M, P'\notin M, T\cap M\neq\emptyset\},$$
We first define a bijection between $A_{\ov{S}}$ and $A_{\ov{T}}$ and then we define an injection from $A_S$ to $A_T$.  For each case, we will define a replacement function $r_i$ on the blocks of $B^+$ and then an injection $f_i$ on the appropriate matchings.  In the end we will have
$$A_{\ov{S}}\overset{f_1}{\longleftrightarrow}A_{\ov{T}}$$
$$A_{S} \overset{f_2}{\hookrightarrow} A_T.$$
Since each map will send matchings of size $k$ to matchings of size $k$, we conclude $m_k(B')\leq m_k(B)$ for all $k \geq 0$.


\emph{Case 1:}   Suppose $M\in A_{\ov{S}}$, that is, $M$ is a matching in $B^+$ such that $P'\in M$, $P\notin M$, and $S\cap M =\emptyset$.  Define $r_1: B^+\to B^+$ by
$$r_1(a,b):=\begin{cases}
(i,j) &\text{ if } a = i', b = j'\\
(i',b) &\text{ if } a = i\\
(a,j') &\text{ if } b = j\\
(a,b)  &\text{ otherwise}
\end{cases}.$$

We can think of $r_1$ as sending the rook at $P'$ to $P$, and then projecting rooks in row $i$ and column $j$ onto row $i'$ and column $j'$ respectively.
We claim $r_1(a,b)\in B^+$ for all $(a,b)\in B^+$.
Since $P = (i,j)$ is an out-corner of $B$, if $(i,b)\in B^+$ then $b\leq j<j'$.
Because $B^+$ is a Young diagram and $P' = (i',j')\in B^+$, we have $r_1(i,b) = (i',b)\in B^+$.
In this case $M\cap S = \emptyset$ so a rook of the form $(a,j)$ with $a\neq i$ must have $a<i'$.
Since $B^+$ is a Young diagram and $P' = (i',j')\in B^+$, we know $r_1(a,j) = (a,j')$ is in $B^+$.
Therefore, $r_1(a,b)\in B^+$ for all $(a,b)\in B^+$.

Define $f_1: A_{\ov{S}} \to A_{\ov{T}}$ by
$$f_1(M) = \{r_1(a,b): (a,b)\in M\}.$$
 First, we show that given $M\in A_{\ov{S}}$, $f_1(M)$ is in fact in $A_{\ov{T}}$.  Sending the rook in $P'$ to $P$ causes conflicts only for rooks in row $i$ and column $j$.   We have solved the problem of a rook in row $i$ since we changed the row of this rook to $i'$.  Note that row $i'$ is otherwise unoccupied since $P'\in M$ and $M$ is a matching.  Similarly, we have solved the problem of a rook in column $j$ since we changed the column of this rook to $j'$. This column is otherwise unoccupied since $M$ is a matching and $P'\in M$.  Thus, $f_1(M)$ is a matching in $B$.
 Moreover, $f_1(M)$ has no rooks in $T$ and so $f_1(M) \in A_{\ov{T}}$.

In addition, $f_1$ is injective.  Define $r_1': B^+ \to B^+$ by
$$r_1'(a,b):=\begin{cases}
(i',j') &\text{ if } (a,b) = (i,j)\\
(i,b) &\text{ if } a = i', b\neq j\\
(a,j) &\text{ if } a\neq i, b = j'\\
(a,b)  &\text{otherwise}
\end{cases}$$
and define $f_1': A_{\ov{T}} \to A_{\ov{S}}$ by
$$f_1'(M) = \{r_1'(a,b): (a,b)\in M\}.$$
 It is straightforward to check that $f_1'(f_1(M))= M$ for all $M\in A_{\ov{S}}$ and $f_1(f_1'(M))= M$ for all $M\in A_{\ov{T}}$.  Thus, there is a bijection between the matchings in $B^+$ with $P'$ and no rooks in $S$ and matchings in $B^+$ with $P$ and no rooks in $T$.

\emph{Case 2:}   Suppose $M\in A_S$.
That is, $M$ is a matching in $B^+$ with a rook in $P'$ and a rook in $S$.
Define $E:= \{(a,b): (a,b)\in B^+, a>i', b>j\}$.
In Figure \ref{boardB+} this is collection of blocks that are both to the right of $P$ and below $P'$.
Let $S^*(M)\subset \{i'+1,\dots,i-1\}$ be the rows of blocks in $S$ that do not share a row with any rooks of $M$ that are in $E$.
Similarly, let $T^*(M)\subset\{j+1,\dots,j'-1\}$ be the columns of blocks in $T$ that do not share a column with any rooks of $M$ that are in $E$.
We claim that $\abs{S^*(M)}<\abs{T^*(M)}$.  Note that $\abs{S} = i-i'-2$ and $\abs{T} = j'-j-2$.  Since $i+j< i'+j'$ we have $i-i'< j-j'$ and so $\abs{S}< \abs{T}$. Letting $a$ be the number of rooks in $E$, then $\abs{S^*(M)} = \abs{S} - a$ and $\abs{T^*(M)} = \abs{T} - a$.  So $\abs{S^*(M)}< \abs{T^*(M)}$.
Thus, there is an injection $s: S^*(M)\to T^*(M)$.
We fix some arbitrary injection $s$ and let $(a_0,j)$ be the location of the rook in $S$. Define  $r_2: B^+ \to B^+$ by
$$r_2(a,b):=
\begin{cases}
(i,j) &\text{ if } (a,b) = (i',j')\\
(i',s(a))  &\text{ if } (a,b)\in S\\
(a_0,b) &\text{ if } a=i\\
(a,j') &\text{ if } b=s(a_0)\\
(a,b) &\text{otherwise}
\end{cases}.$$

In Figure \ref{r2sketch} the gray boxes are the images of the black boxes under the map $r_2$.  We can think of $r_2$ as sending the rook in $P'$ to $P$ (arrow 1), sending the rook in $S$ to a place in $T$ via $s$ (arrow 2), and then projecting rooks in conflicting rows and columns to rows and columns that are known to be unoccupied (arrows 3 and 4).

\begin{center}
\begin{figure}[!h]
\begin{tikzpicture}[scale = 2]

\coordinate (P) at (0,-1);
\coordinate (P') at (2,0);
\coordinate (S) at (0,-.5);
\coordinate (S') at (1,0);
\coordinate (V) at (-1,-1);
\coordinate (V') at (-1,-.5);
\coordinate (H) at (1,.5);
\coordinate (H') at (2,.5);

\draw (-.1,-1.3) node {$P$};
\draw[below right] (P') node {$P'$};
\draw[left] (-2,-.1) node {$i'$};
\draw[left] (-2,-1.1) node {$i$};

\draw [very thick, -stealth] (1.9,-.2) to (0,-1.15);
\draw (1,-.75) node {$1$};
\draw [very thick, -stealth] (0,-.35) to (.7,-.2);
\draw (.35, -.4) node {$2$};
\draw [very thick, -stealth] (-1.3,-1) to (-1.3,-.4);
\draw (-1.4,-.7) node{$3$};
\fill[color=black] (.6,.6) rectangle (.8,.8);
\draw [very thick, -stealth] (.8,.7) to (1.8,.7);
\draw (1.3,.8) node {$4$};

\draw[step=.2] (-.2,-1) grid (0,-.8);
\draw[dotted,very thick] (-.1,-.7) -- (-.1,-.5);
\fill[color=black] (-.2,-.4) rectangle (0,-.2);
\draw[step=.2] (-.2,-.4) grid (0,0);
\draw[step=.2] (-.2,0) grid (0,.4);
\draw[dotted,very thick] (-.1,.5) -- (-.1,.7);
\draw[step=.2] (-.2,.8) grid (0,1);
\draw (-.2,.8) -- (0,.8);

\draw[step=.2] (1.8,0) grid (2,.2);
\draw[dotted,very thick] (1.9,.3) --(1.9,.5);
\fill[color=gray] (1.8,.6) rectangle (2,.8);
\draw[step=.2]  (1.8,.6) grid (2,1);
\draw (1.8,.6) -- (1.8,1);
\draw (1.8,.6) -- (2,.6);
\draw (1.8,0) -- (1.8,.2);

\draw[step=.2] (-2,-.2) grid (-1.2,0);
\draw[dotted, very thick] (-1.1,-.1) -- (-.9,-.1);
\draw[step=.2] (-.8,-.2) grid (0,0);
\fill[color=gray] (.6,-.2) rectangle (.8,0);
\draw[step=.2] (0,-.2) grid (.8,0);
\draw[dotted,very thick] (1,-.1) -- (1.2,-.1);
\fill[color=black] (1.8,-.2) rectangle (2,0);
\draw[step=.2] (1.4,-.2) grid (2,0);
\draw (1.4,-.2) -- (1.4,0);
\fill[color=gray] (-1.4,-.4) rectangle (-1.2 ,-.2);
\draw (-1.4,-.4) rectangle (-1.2 ,-.2);

\fill[color=black] (-1.4,-1.2) rectangle (-1.2,-1);
\draw[step=.2] (-2,-1.2) grid (-1.2,-1);
\draw[dotted,very thick]  (-1,-1.1) -- (-.8,-1.1);
\fill[color=gray] (-.2,-1.2) rectangle (0,-1);
\draw[step=.2] (-.6,-1.2) grid (0,-1);

\end{tikzpicture}
\caption{A sketch of the map $r_2$} \label{r2sketch}
\end{figure}
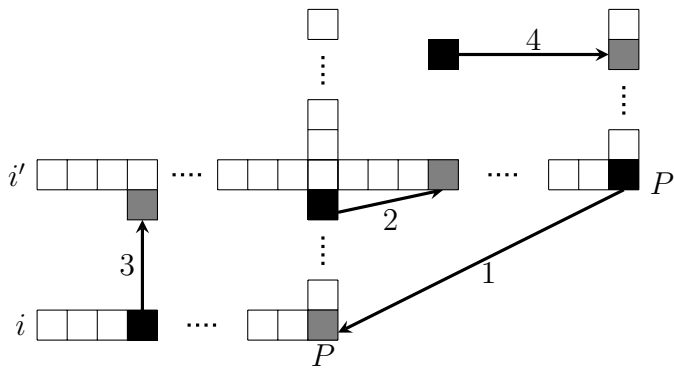
\end{center}

Again we need to show that $r_2(a,b) \in B^+$ for all $(a,b)\in B^+$.  If $(a,b)\in S$ then $(i',s(a))\in B^+$ as $(i',s(a))\in T$.  If the rook $(a,b)$ is in row $i$ then $r_2(a,b) = (a_0,b)\in B^+$ as $a_0 <i$ and $(i,b)\in B^+$.  Finally, if $(a,b) = (a,s(a_0))$, then $r_2(a,b) = (a,j') \in B$ since $a<i'$ and $(i',j')$ is in $B^+$.

For $M\in A_S$, define
$$f_2(M):=\{r_2(a,b): (a,b)\in M\}.$$
We claim that if $M$ is a matching in $A_S$ then $f_2(M)$ is a matching in $A_T$.
First we will show that no two rooks are in the same row.  There are only three rooks that change rows when we apply $r_2$.
These rooks originally have rows $i'$, $i$ and $a_0$.
After applying $r_2$ these rooks are in rows $i$, $a_0$ and $i'$ respectively.
Thus the rooks of $f_2(M)$ occupy the same collection of rows as those of $M$.
Now we must show that no two rooks occupy the same column.
Similarly, there are only three rooks that change columns.
These rooks originally occupy columns $j',j,$ and $s(a_0)$.  After applying $r_2$, these rooks occupy $j,s(a_0),$ and $j'$ respectively.  So the rooks of $f_2(M)$ occupy the same collection of columns as those of $M$ and so no two rooks are in the same column.  Finally, $T\cap f_2(M) \neq \emptyset$ since the rook in $S$ got sent to a rook in $T$. Thus $f_2(M)$ is a matching in $A_T$.

Using $f_1$ and $f_2$ we know where to send all matchings in $B^+$ that have a rook in $P'$ and not $P$.  Moreover, the images of $f_1$ and $f_2$ are disjoint and thus there is an injection from matchings in $B^+$ with a rook in $P'$ and not $P$ to matchings in $B^+$ that have a rook in $P$ and not $P'$.  Since this injection preserves the size of the matching, $m_k(B') \leq m_k(B)$.

It remains to prove that $B'$ has strictly fewer matchings than $B$.  Suppose that there are no rooks of $M$ that are in $E$ and fix an injection $s: S^*(M)\to T^*(M)$  Since $\abs{S^*(M)}<\abs{T^*(M)}$ there exists $Q\in T^*(M)$ such that $Q\neq s(R)$ for any $R\in S^*(M)$.  Then $\{Q,P\}\in A_T$ and there does not exist $M\in A_S$ such that $f_2(M) = \{Q,P\}$.  Thus, $m(B')<m(B)$.
\end{proof}

\begin{defi}
Given a Young diagram $B$ with an out-corner $P$ and an in-corner $P'$ with $s(P)<s(P')$ we call the move that results in $B-P+P'$ an \emph{out-block move}.
\end{defi}

There are examples of Young diagrams that are not the Young diagrams associated to $L_{\ell,r}(e)$ that have no out-block moves.
For example, see Figure \ref{YDNOB}.  For this reason we are forced to introduce an additional move.
It is clear that taking the transpose of a Young diagram preserves $m_k(B)$ for all $k$.
The following definition describes a way that we can transpose a piece of a Young diagram.

\begin{defi}
Let $B$ be a Young diagram in an $\ell\times r$ frame and let $(i,j)\in B$.  Define the \emph{transpose of $B$ at $P = (i,j)$} to be the diagram
$$B_{P}^* := \{(a,b)\in B: a<i \text{ or }  b<j\} \cup  \{ (a,b)^*: (a,b)\in B, a\geq i, \text{ and } b\geq j\}$$
where $(a,b)^* = (b-j+i,a-i+j)$.  Call transposing at $P$ {\em legal} if $B_{P}^*$ is a Young diagram in an $\ell\times r$ frame.
\end{defi}

Note that the definition of $(a,b)^*$ depends on $(i,j)$, the place we are transposing, but we suppress this in the notation.

\begin{ex} Figures \ref{YDNOB} and \ref{Transposed} are two Young diagrams in a $4\times 5$ frame.  To get the Young diagram in Figure \ref{Transposed} from the Young diagram in Figure \ref{YDNOB} we transpose at $P= (1,2)$

\begin{figure}[h!]
\begin{multicols}{2}
\begin{center}
\begin{tikzpicture}[scale = .5]
  \draw[color=gray] (0,0) rectangle (5,4);
\draw (0,0) grid (3,4);
 \fill [color=black] (1.5,3.5) circle (3pt);
\end{tikzpicture}
\caption{The Young diagram $B$ of (3,3,3,3)}\label{YDNOB}

\begin{tikzpicture}[scale = .5]
\draw[color=gray] (0,0) rectangle (5,4);
\draw (0,2) grid (5,4);
\draw (0,0) grid (1,2);
\end{tikzpicture}
\caption{The Young diagram of $B^*_{(1,2)} = (5,5,1,1)$}\label{Transposed}
\end{center}
\end{multicols}
\end{figure}

\end{ex}

We will now prove that legally transposing at $P\in B$ preserves $m_k$ for all $k\geq 0$.

\begin{lem}\label{transpose}
Let $B$ be a Young diagram and let $P\in B$.  Performing a legal transpose at $P$ preserves the number of matchings of all sizes.
\end{lem}

\begin{proof}
Suppose that the sub-board to be transposed has dimensions $x\times y$.  Write $P = (i,j)$.  Without loss of generality, let $\max\{x,y\}=x$.  We single out the following pieces of $B$ as sketched in Figure \ref{SixPieces} :
\begin{align*}
    T &= \setof{(a,b)}{(a,b)\in B, a\geq i, b\geq j}   &&\text{the \bf{T}ransposed portion}\\
    U &= \setof{(a,b)}{ a<i,j\leq b \leq j+x} &&\text{the portion \bf{U}p from $T$}\\
    L &= \setof{(a,b)}{i\leq a \leq i+x, b<j} &&\text{the portion to the \bf{L}eft of $T$}\\
\end{align*}

\begin{center}
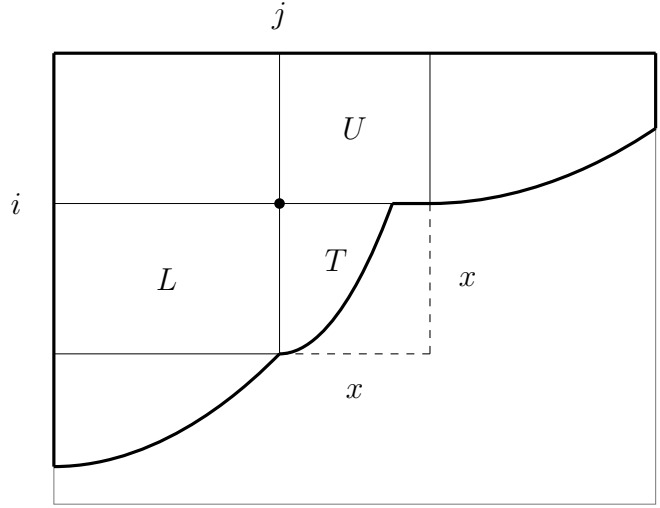
\begin{figure}
\begin{tikzpicture}
\draw[color = gray] (0,0) rectangle (8,6);

\draw[very thick] (0,.5) parabola (3,2);
\draw[very thick] (3,2) parabola (4.5,4);
\draw[very thick] (5,4) parabola (8,5);
\draw[very thick] (0,.5) -- (0,6);
\draw[very thick] (0,6) -- (8,6);
\draw[very thick] (8,6) -- (8,5);
\draw[very thick] (4.5,4) -- (5,4);

\draw[-,very thin] (0,2) -- (3,2);
\draw[-,very thin] (3,2) -- (3,6);
\draw[-,very thin] (0,4) -- (5,4);
\draw[-,very thin] (5,4) -- (5,6);

\draw[dashed] (3,2) -- (5,2);
\draw (4,1.5) node {$x$};
\draw[dashed] (5,2) -- (5,4);
\draw (5.5,3) node {$x$};

\draw (4,5) node {$U$};
\draw (1.5,3) node {$L$};
\draw (3.75,3.25) node {$T$};
\draw (-.5,4) node {$i$};
\draw (3,6.5) node {$j$};

\fill (3,4) circle (2pt);

\end{tikzpicture}
\caption{Pieces of $B$.}\label{SixPieces}
\end{figure}
\end{center}

Fix a matching in $T$, call it $M_T$.  Define $T^*:=\{(a,b)^*: (a,b)\in T\}$ and let $M_T^*$ be the matching in $T^*$ obtained by transposing the location of each of the rooks. We will define an injection from matchings in $B$ that contain $M_T$ to matchings in $B^*_P$ that contain $M_T^*$.  Doing this for every matching in $T$ will show that there are at most as many matchings in $B$ as in $B_{P}^*$.  A similarly defined injection will work to conclude that $B_{P}^*$ has at most as many matchings as $B$.  Since these injections will preserve the size of each matching we will conclude that $m_k(B) = m_k(B_P^*)$.

Let $U_1\subseteq \{j,\dots, j+r\}$ be the set of column indices between $j$ and $j+r$ that are unoccupied by a rook in $M_T$ and similarly let $U_2\subseteq\{j,\dots, j+x\}$ be the set of column indices between $j$ and $j+x$ that are unoccupied by a rook in $M_T^*$.
If $\abs{M_T}=t$ then $\abs{U_1} = \abs{U_2} = x+1 - t$.
Thus, there is a bijection $u: U_1\to U_2$.

Similarly, let $L_1\subseteq \{i,\dots,i+x\}$ be the set of row indices between $i$ and $i+x$ that are not occupied by a rook in $M_T$ and let $L_2\subseteq\{i,\dots, i+x\}$ be the set of row indices between $i$ and $i+x$ of rows that are not occupied by a rook in $M_T^*$.
Again, $\abs{L_1}=\abs{L_2} = x+1 - t$ where $t$ is the size of $M_T$ and hence there is a bijection $l: L_1 \to L_2$.

 Define $r: B \to B^*_P$ by
$$r(a,b):=
\begin{cases}
(a,b)^* &\text{ if } (a,b) \in T\\
(l(a),b)  &\text{ if } (a,b)\in L, a \in L_1\\
(a,u(b)) &\text{ if } (a,b) \in U, b \in U_2\\
(a,b) &\text{ otherwise}
\end{cases}.$$

Define $f$ from matchings in $B$ containing $M_T$ to matchings in $B^*_P$ containing $M_T^*$ by
$$f(M) = \{ r(a,b): (a,b) \in M\}.$$
  First we note that given $M\in \M(B)$ we have $f(M) \in \M(B^*_P)$.  We know $f(M) \subset B^*_P$ since $r(a,b)\in B^*_P$ for all $(a,b) \in B$.  For each rook, $(a,b)\in M_T$ we send $(a,b)$ to $(a,b)^*$.  Conflicts are only caused in rows $i,i+1,\dots, i+r$ and columns $j,j+1,\dots j+r$.  These conflicts are resolved using the injections $l$ and $u$ which send all rooks in $L$ and $U$ to rows and columns unoccupied by $M_T^*$.  This causes no additional conflicts since no additional rows or columns are changed.  Thus $f(M) \in \M(B^*_P)$.  Moreover, for a matching $M\in \M(B)$ containing $M_T$, $f(M)$ contains $M_T^*$.

We claim that $f$ is a bijection.  Define
$$r'(a,b):=
\begin{cases}
(a,b)^* &\text{ if } (a,b) \in T\\
(l^{-1}(a),b) &\text{ if } (a,b)\in L, a \in L_2\\
(a,u^{-1}(b)) &\text{ if } (a,b) \in U, b \in U_2\\
(a,b) &\text{ otherwise}
\end{cases}.$$
Define $f'$ from matchings in $B^*_P$ containing $M_T$ to matchings in $B$ containing $M_{T}^*$ by
$$f'(M) = \{r'(a,b): (a,b)\in M\}.$$
It is straightforward to check that $f'$ is actually the inverse of $f$.  Thus, $f$ is a bijection and $m_k(B) = m_k(B^*)$.
\end{proof}

The next lemma shows how we piece together the out-block move and the transpose move.  First we define the lex order on Young diagrams.

\begin{defi}
To define the lex order on Young diagrams, we first define an ordering on ordered pairs.  Say $(a,b)\lesssim(c,d)$ if $a<c$ or $a=c$ and $b<d$.  Then the \emph{lex order} $<_L$ on Young diagrams is defined by $B<_L B'$ if and only if $\min_{\lesssim}(B\Delta B')\in B$.
\end{defi}

Note that by this definition, $L_{\ell,r}(e)$ is least among all Young diagrams in $\ell\times r$ frames.  The next lemma states that if we can't find an out-block move we can find a legal transpose that moves a board to one earlier in lex order.

\begin{lem}\label{transposeupinlex}
Consider a Young diagram $B$ in an $\ell\times r$ frame where $\ell\leq r$ that has no out-block moves and is not the Young diagram of $L_{\ell,r}(e)$. There exists $P\in B$ such that the transpose at $P$ is legal and $B^*_P<_L B$.
\end{lem}

\begin{proof}
First we set up some notation.
For $P = (i,j)\in B$, let $\rho(P) = i$ (so $\rho(P)$ is the row of $P$) and let $c(P) = j$ (so $c(P)$ is the column of $P$).
For $P,Q\in B$, define $v(P,Q) = \abs{\rho(P) - \rho(Q)}$, so $v(P,Q)$ is the vertical distance between $P$ and $Q$.
In the Young diagram $B$ let $P$ be the out-corner with $\rho(P)$ greatest among all out-corners and $Q$ be the in-corner of $B$ with $\rho(Q)$ least among all in-corners.
If legal, transpose at $S=(\rho(Q),c(P))$.
If transposing at $S$ is not legal then we ``count back" from the right hand limit of the partition so that the vertical distance between $P$ and $Q$ will fit.
In more detail, transpose at $S' = (\rho(Q), r - v(P,Q))$ whenever transposing at $S$ is not legal.
We claim that this gives a place to transpose legally and that it results in a board that is earlier in lex order than $B$.
Call the result of the performed transpose $B^*$.

First suppose that transposing at $S$ is legal. Note that $S\in B$ as $B$ is a Young diagram and $P\in B$.
If we transpose at $S$ then it is because the transpose at $S$ is legal.
The first row where $B^*$ is different from $B$ is $\rho(Q)$.
In $B$ this row has $c(Q)-1$ blocks.
In $B^*$ this row has $c(P) + v(P,Q) = c(P) + \rho(P) - \rho(Q)$ blocks.
(Since $B$ is not the Young diagram of $L_{\ell,r}(e)$ we know $\rho(P)>\rho(Q)$ and so may remove the absolute value signs in $v(P,Q)$.)
Since $B$ has no out-block moves, we know $s(P)>s(Q)$, i.e, $\rho(P) + c(P) > \rho(Q) + c(Q)$.
Thus $c(Q) < c(P) + \rho(Q) - \rho(Q)$ and so $B^*>_L B$.

Now suppose that transposing at $S$ is not legal.
We claim that that $S' =(\rho(Q), r- v(P,Q))\in B$.
Because transposing at $S$ is not legal, we know $c(Q) + v(P,Q) >r$.  Thus, $r-v(P,Q)< c(Q)$.
Moreover $r-v(P,Q)\geq 1$ as $\ell\leq r$ and $v(P,Q)<\ell$.
Using that $Q$ is an in-corner of $B$, we get that $S'\in B$.
We need to show that the transpose is legal.
Intuitively, the transposed piece fits horizontally because we ``counted back" far enough to make it so.
Algebraically, the length of the row $\rho(Q)$ in $B^*$ is $r-v(P,Q) + v(P,Q) = r$ which is exactly the length that can fit.
We'll show that the transposed piece fits vertically after we show that $B^*_S>_L B$. Note the first place $B$ and $B^*$ differ is in row $\rho(Q)$ and row $\rho(Q)$ in $B^*$ has $r$ blocks while row $\rho(Q)$ in $B$ has strictly fewer than $r$ blocks since $Q$ is an in-corner.
In particular, this means, if the dimensions of the transpose section are $s\times t$, we have $s\geq t$.  Thus the transposed piece will fit vertically.
Therefore, transposing at one of $S$ or $S'$ will result in a legal transpose that moves $B$ earlier in lex order.

\end{proof}

We are now ready to prove Theorem~\ref{bipartite}.  The proof follows easily from Lemmas \ref{differencegraph} and \ref{transposeupinlex}.

\begin{proof}[Proof of Theorem~\ref{bipartite}.]  Suppose $G$ is a bipartite graph that is not equal to $L_{\ell,r}(e)$.  If $G$ is not bipartite threshold then we can apply Lemma~\ref{differencegraph} to get another graph with the same number of vertices and edges, but at most as many matchings.  If $G$ is bipartite threshold, consider the associated Young diagram $B_G$.  If $B_G$ has no out-block moves and $G\neq L_{l,r}(e)$ then by Lemma~\ref{transposeupinlex} there is a legal transpose at $(i,j)$ that moves the associated board earlier in lex order.  This move preserves the number of matchings by Lemma~\ref{transpose}.  Thus, $L_{l,r}(e)$ attains the minimum number of matchings of all sizes.  This concludes the proof of Theorem~\ref{bipartite}.
\end{proof}

\section{Proof of Theorem~\ref{general}}

Recall that our main result states that either the lex or colex graph minimizes the number of matchings (of size $k$) in $\cG_{n,m}$.  In this section we will prove this.  First we will show that there is a graph attaining the minimum that is threshold.  Finally, we will use  the bipartite case to complete the proof.

\begin{lem}\label{compressing}
For all graphs $G$ and all $x,y\in V(G)$
$$m_k(G_{x\to y}) \leq m_k(G).$$
\end{lem}
\begin{proof}
Let $H:=G_{x\to y}$.  As in the proof of Lemma~\ref{differencegraph} we will construct an injection $\phi$ from $\M_k(H)\setminus \M_k(G)$ to $\M_k(G)\setminus \M_k(H)$ from which it follows that $m(H)\leq m(G)$.  Let $A = E(x,N_G(x,\ov{y}))\subset E(G)$ and $B=E(y,N_G(x,\ov{y}))\subset E(H)$.
Then $H=G-A+B$. So
$$\M_k(H)\setminus \M_k(G)=\{M\in \M_k(H): M\cap B\neq \emptyset\}$$ and similarly
$$\M_k(G)\setminus \M_k(H)=\{M\in \M_k(G): M\cap A\neq \emptyset\}$$
 Define a replacement function $r:E(H)\to E(G)$.  Let
$$r(e):=\begin{cases}
e \Delta\{x,y\} &\text{ if } e\in B \\
yz  &\text{ if } e=xz, z\neq y\\
e  &\text{ otherwise }
\end{cases}.$$

For each $e\in E(H)$ note that $r(e)\in E(G)$.
If $e\in B$ then $r(e) = e\Delta \{x,y\}\in A\subset E(G)$.
If $e=xz\in H$ and $z\neq y$ then $z\in N_G(x)\cap N_G(y)$ and so $\phi(e) = yz\in E(G)$.  Finally, if $e\notin B$ then $e\in E(G)\cap E(H)$.

Now define $\phi:\M_k(H)\setminus \M_k(G)\to \M_k(G)\setminus \M_k(H)$   by
$$\phi(M):= \{r(e): e\in M\}.$$

Suppose $M\in \M_k(H)\setminus \M_k(G)$.
We claim $\phi(M)\in \M_k(G)\setminus \M_k(H)$.
Since $r:E(H)\to E(G)$ we know $\phi(M)\subset E(G)$.
To show that $\phi(M)\in \M_k(G)$ we suppose to a contradiction that $r(e)$ is incident to $r(f)$ for some $e,f\in M$.
For the first case, suppose that $r(e) \cap r(f) = x$.
Note that $x \in r(e)$ for any edge $e$ if and only if $e \in B$ or $e = xy$.  Since $M \cap B \neq \emptyset$ we know that $xy\notin M$.
Thus, both $e$ and $f$ are in $B$ and $e\cap f = y$, a contradiction since $e$ and $f$ are in the matching $M$.

Next suppose that $r(e) \cap r(f) = y$.
In $H$ we know neither $e$ nor $f$ are incident to $y$ since $M\cap B \neq \emptyset$ and if $e$ or $f$ are in $B$ then we replaced them with edges incident to $x$.
Thus, it must be the case that both $e$ and $f$ are incident to $x$ in $M$, a contradiction since $M$ is a matching.
Finally, if $r(e) \cap r(f) = z$ for $z\neq x$ and $z\neq y$ then $r$ acted as the identity on $e$ and $f$. So $e\cap f = z$, a contradiction.
Thus, $\phi(M) \in \M_k(G)$.
Finally, $\phi(M)\notin \M_k(H)$ since $\phi(M)$ has an edge from $A$ and $A\cap E(H) = \emptyset$.

To complete the proof we show $\phi$ is an injection.  Define $r':E(G)\to E(H)$ by
$$r'(e):=
\begin{cases}
e\Delta\{x,y\} &\text{ if } e\in A\\
xz &\text{ if } e=yz, z\neq x\\
e &\text{otherwise}
\end{cases}.$$
Define $\phi': \im(\phi)\to \M_k(H)$ by $\phi'(M) = \{r'(e): e\in M\}$. Given $M\in\M_k(H)\setminus\M_k(G)$ it is easy to check that $\phi'(\phi(M)) = M$.  Therefore $\phi$ has a left inverse and so $\phi$ is injective.

\end{proof}

Thus we may assume that the graph that minimizes $m_k$ is threshold.  In fact, we can find a graph that minimizes $m_k$ that has even more structure.

\begin{defi}
For a threshold graph $G$ write $V(G) = K\cup I$ where $G[K]$ is a clique and $G[I]$ is an independent set.  Let $B$ be the bipartite graph with partite sets $K$ and $I$ and edge set $E(K,I)$.  If $E(K,I)$ is $\cL_{|K|,|I|}(e)$ for some $e$, we will say that $G$ is \emph{lex-across}.
\end{defi}

We will now prove a lemma that states that if some parameter is either maximized or minimized by a lex-across graph, then that parameter is maximized or minimized by the lex or colex graph.

\begin{lem}\label{lem:lexacross}
If a parameter $P$ is such that for all $n,e$ there is a $P$-optimal graph in $\cG_{n,e}$ that is lex-across, then  either the lex graph $\cL(n,e)$ or the colex graph $\cC(n,e)$ is $P$-optimal.
\end{lem}

\begin{proof}
Given $n,e$ let $\cO$ be the collection of all triples $(G,K,I)$ such that $V(G) = K\cup I$, $G[K]$ is complete, $G[I]$ is independent, and $G$ is a $P$-optimal lex-across graph. Define $\ov{s}$ to be the maximum size of $K$ among all triples in $\cO$ and let $\und{s}$ be the minimum size of $K$ among all triples in $\cO$.  We consider two cases: when $\ov{s}\geq \frac{n}{2}$ and when $\ov{s}< \frac{n}{2}$.

Suppose first that $\ov{s}\geq \frac{n}{2}$.  Let $(G,K,I)\in \cO$ such that $\abs{K} = \ov{s}$.  In this case $K$ has at least as many vertices as $I$. If $v\in I$ has $N(v) = K$ then $(G, K\cup\{v\}, I\setminus\{v\}) \in \cO$ and $\abs{K\cup\{v\}}\geq \ov{s}$, a contradiction. Since $(G,K,I)\in \cO$ we know $G$ is lex-across, and so it must be the case that some vertex in $I$ has fewer than $\ov{s}$ neighbors and all other vertices in $I$ are isolated.  Thus $G$ is the colex graph $\cC(n,e)$.

Suppose now that $\ov{s}< \frac{n}{2}$.  In this case, $\und{s}<\frac{n}{2}$ as well.  Let $(G,K,I) \in \cO$ such that $\abs{K} = \und{s}$.  Here $\abs{K}<\abs{I}$.  Suppose that there is a vertex $k\in K$ such that $N(k) \cap I = \emptyset$.  Then $(G,K\setminus\{k\}, I\cup \{k\})\in \cO$ and $\abs{K\setminus\{k\}}<\und{s}$, a contradiction. So every vertex in $K$ is adjacent to some vertex in $I$.  Since $G$ is lex across all vertices in $K$ except for one possible exception have closed neighborhood all of $G$.  Thus $G$ is the lex graph $\cL(n,e)$.
\end{proof}

We are now ready to prove our main theorem, Theorem~\ref{general}.

\begin{proof}[Proof of Theorem~\ref{general}.]
Let $G\in \cG_{n,e}$ be a graph that minimizes $m$ (respectively $m_k$). By Corollary~\ref{cor:allcompressed2} and Lemma~\ref{compressing} we can assume that $G$ is threshold.
By Theorem~\ref{bipartite} the lex bipartite graph minimizes $m_k$ for all $k\geq 0$ in $\cB_{\ell,r,e}$, so by Lemma~\ref{minimize} we can assume $G$ is lex-across.
By Lemma~\ref{lem:lexacross} the lex or colex graph minimizes $m$ (respectively $m_k$) in $\cG_{n,e}$ which completes the proof of Theorem~\ref{general}.
\end{proof}

\section{Further Directions}

There are many open problems remaining in this area. For instance the Upper Matching Conjecture of Friedland, Krop, and Markstr\"om \cite{FKM} claims that for all $d$-regular graphs $G$ on $2n$-vertices such that $d$ divides $n$ we have 
\[
  m_k(G) \le m_k\bigl(\frac{n}d K_{d,d}\bigr)
\]
for all $k$.

We know that the lex or colex graph doesn't necessarily minimize $m_k(G)$ for all $k$ simultaneously.  For example, consider the family $\cG_{18,87}$.  Then
$$\begin{array}{c | c | c}
 &m_2  &m_7\\
 \hline
 \cL(18,87) &2745 &0\\
 \hline
 \cC(18,87) &2739 &93,555
\end{array}$$
While $m_2(\cC(18,87)) < m_2(\cL(18,87))$, the lex graph has no $7$-matchings and the colex graph has many. This indicates that it is a non-trivial problem to determine the graph $G\in \cG_{n,e}$ that minimizes the matching polynomial
\[
	m_G(\lambda) = \sum_{k\ge 0} m_k(G) \lambda^k
\]
for a given value of $\lambda>0$. Theorem~\ref{general} includes case $\lambda=1$. By Lemma~\ref{compressing} the extremal graph can be taken to be threshold.

\bibliographystyle{amsplain}
\bibliography{Matchings}

\end{document}